\documentclass[12pt]{article}
\usepackage[utf8]{inputenc}
\usepackage{amsmath}
\usepackage[T2A]{fontenc}
\usepackage[english]{babel}
\usepackage{hyphenat}
\usepackage{floatrow}

\usepackage{graphicx}
\graphicspath{ {./Images/} }
\usepackage{array}
\usepackage{caption}
\usepackage{amsfonts, amsthm, xcolor}
\usepackage{url}
\usepackage[margin=1in, total={8.5in, 11in}]{geometry}
\usepackage{tikz}
\usetikzlibrary{decorations.pathreplacing}

\newcommand{\R}{\mathbb{R}}

\newcommand{\Z}{\mathbb{Z}}
\newcommand{\F}{\mathcal{F}}
\newcommand{\M}{\mathcal{M}}
\newcommand{\Sf}{\mathcal{S}}
\newcommand{\E}{\mathbb{E}}
\newcommand{\Pp}{\mathbb{P}}
\newcommand{\m}{\mathcal}
\newcommand{\ff}{\mathcal{F}}
 
\title{Satisfying sequences for rainbow partite matchings}
\author{Andrey Kupavskii\footnote{Moscow Institute of Physics and Technology, St. Petersburg State University, Innopolis University, Russia; Email: {\tt kupavskii@ya.ru}}, Elizaveta Popova\footnote{Weizmann Institute of Science, Rehovot, Israel; Email: {\tt elizaveta.popova@weizmann.ac.il}}}

\newtheorem{thm}{Theorem}
\newtheorem{lem}[thm]{Lemma}

\newtheorem{pro}[thm]{Problem}

\newtheorem{cla}[thm]{Claim}
\newtheorem{conj}[thm]{Conjecture}

\begin{document}

\maketitle
\begin{abstract}
    Let $\mathcal F_1,\ldots, \mathcal F_s\subset [n]^k$ be a collection of $s$ families. In this paper, we address the following question: for which sequences $f_1,\ldots, f_s$ do the conditions $|\ff_i|>f_i$ imply that the families contain a rainbow matching, that is, there are pairwise disjoint $F_1\in \ff_1,\ldots, F_s\in \ff_s$? We call such sequences {\em satisfying}. Kiselev and the first author showed that $f_1 = \ldots = f_s=(s-1)n^{k-1}$ is satisfying for $s>470$, verifying the conjecture by Aharoni and Howard. This is the best possible if we insist that the families are of the same size. However, it turns out that much more can be said about asymmetric restrictions. In this paper, we investigate this question in several regimes and in particular answer the questions asked by Kiselev and Kupavskii. We use a variety of methods, including concentration and anticoncentration results, spread approximations, and the Combinatorial Nullstellenzatz. 
\end{abstract}
\section{Introduction}
Let $[n] = \{1,\ldots, n\}$ be the standard $n$-element set and, more generally, $[a,b] = \{a,a+1,\ldots, b\}$. For a set $X$, let $2^{X}$ and ${X\choose k}$ stand for the power set of $X$ and the set of all its $k$-element subsets, respectively. We denote by $[n]^k$ the collection of all tuples $(x_1,\ldots, x_k)$, where $x_i\in [n]$. That is, $[n]^k$ is a complete $k$-partite $k$-uniform hypergraph with parts of size $n$ (and so it has $kn$ vertices).  

For an integer $s$, consider $s$ families  $\F_1,..., \F_s \subset [n]^k$. A {\it rainbow matching} or a {\it cross-matching} is a collection of $s$ sets $F_1\in \ff_1,\ldots, F_s\in \ff_s$ that are pairwise disjoint. If $\ff_1,\ldots, \ff_s$ do not contain a cross-matching then we call them {\it cross-dependent}. Note that if $s>n$, then the families are necessarily cross-dependent. In what follows, we assume that $s\le n.$ 

In this paper, we investigate which lower bounds on the sizes of $\ff_i$ guarantee the existence of a cross-matching. We call a sequence $f_1,\ldots,f_s$ {\it satisfying} if the conditions $|\ff_i|>f_i$ for each $i\in[s]$ imply that the families $\ff_1,\ldots, \ff_s$ contain a rainbow matching. Note that this notion depends on $k$, but we omit it since $k$ is always clear from the context. Let $(f_1,\ldots, f_s),$ $(f_1',\ldots, f_s')$ be two sequences of $s$ real numbers, such that $f_i\le f_i'$ for all $i\in[s]$. In this case, we write $(f_1,\ldots, f_s)\preceq (f_1',\ldots, f_s')$. It is then straightforward to see that if $(f_1,\ldots, f_s)$ is satisfying, then $(f_1',\ldots, f_s')$ is satisfying as well. We shall sometimes write $\{f_i\}_i$ instead of $(f_1,\ldots, f_s)$. We shall write $f_i$ in increasing order.

One of the classical questions in extremal set theory is the Erd\H os Matching Conjecture (EMC) \cite{E} regarding the size of the largest family in ${[n]\choose k}$ that contains no $s$ pairwise disjoint sets (an {\it $s$-matching}). In spite of the efforts by different researchers, it is still open in general. For the best current results, see \cite{F4}, \cite{FK21}, \cite{KK}. Its $k$-partite analogue for families $[n]^k$ states that any family of size strictly larger than $(s-1)n^{k-1}$ has an $s$-matching. Unlike for ${[n]\choose k},$ where it is a major open problem, the statement for $[n]^k$ is very easy to prove via averaging. The EMC has a multifamily version, and so does its $k$-partite analogue. Let us state the latter: given $s$ families $\ff_1,\ldots, \ff_s$ such that  $|\F_i| > (s-1)n^{k-1}$ for all $i \in [s]$, these families contain a cross-matching. This is known as the Aharoni--Howard conjecture \cite{AH}. Unlike the one-family version, this does not have an easy averaging proof and was open for some time, until Kiselev and the first author \cite{KKconc} resolved it for all $s>470$ (and any $n,k$). 

The value  $(s-1)n^{k-1}$ in the Aharoni--Howard conjecture is optimal if we impose the same constraint on all families. However, it became apparent in the paper \cite{KKconc} that the problem has a certain asymmetric nature, and the resolution of the Aharoni--Howard conjecture relied on this asymmetricity. Kiselev and the first author introduced the notion of a satisfying sequence and noted that, in particular, $\{(i + C\sqrt{s\log s}) n^{k-1}\}_i$ for some absolute $C$ is satisfying and suggested doing a more thorough investigation of the matter. 

A naive guess could be that even the sequence $f_i=(i-1)n^{k-1}$ is satisfying, cf. \cite{KKconc}. However, in \cite{KKconc} the authors show that even the sequence $f_1 = \ldots = f_{s-1}=(s-1)n^{k-1}-(n-1)^{k-1}-1$ and $f_s = (s-1)n^{k-1}$ is not satisfying. For large $k$, this sequence has form $f_i = (s-1-o(1))n^{k-1}$. They put forth the following conjectures about the sequences they believe to be indeed satisfying.

\begin{conj}[\cite{KKconc}]
    The following sequences are satisfying (for $s \leq n$):
    \begin{enumerate}
        \item $f_i = i n^{k-1}$, 
        \item $f_i = \min\{i + C\sqrt{s\log s}, s-1\} n^{k-1}$.
    \end{enumerate} 
\end{conj}

In this paper, we investigate these conjectures. It turns out that the second sequence is not always satisfying, and we manage to delineate the regime in which it is. As for the first sequence, we show that it is satisfying for somewhat large $n$. At the same time, we give an example showing that the sequence $f_i = (i-1/2) n^{k-1}-1$ is not always satisfying. 

\begin{thm}\label{thmmain2}
    For $n \geq \max\{2^{8}s^{3/2} \log_2^{3/2}(sk), 8 s^2\}$ the sequence $f_i = i\cdot n^{k-1}$ is satisfying.
\end{thm}

We prove a stronger version of this result (Theorem~\ref{thmmain3}) in Section~\ref{sec5} using the method of spread approximation, introduced in \cite{KuZa}. Since then, it has been used to obtain progress in several extremal set theory problems, see e.g. \cite{Kup54, Kup55, Kup56, KN}.

\begin{thm}\label{thmmain1}
    The sequence $f_i = \min(s-1, i + C\sqrt{s\log s})n^{k-1}$ is satisfying for $C \geq 20$, $s > s_0(C)$ with some $s_0$ depending only on $C$, and $n > \frac{3k}{C}\sqrt{s/\log s}$.
\end{thm}
In Section~\ref{sec2} we give an example showing that this bound on $n$ is optimal up to constants, unlike the lower bound on $n$ in Theorem~\ref{thmmain2} which is likely to be due to the restrictions of our proof and not the violation of the conjecture. The proof of this theorem combines three major ingredients. The first one is an adaptation of the original approach from \cite{KKconc} that is based on a certain concentration phenomenon (cf. Theorem~\ref{conc}) for the random variable $\zeta_\ff = |\ff\cap \mathcal M|$, where $\ff\subset [n]^k$ is a family under study, and $\mathcal M$ is a uniformly random perfect matching in $[n]^k$. The second ingredient is an anticoncentration result (Theorem~\ref{antic}),  presented in Section~\ref{sec3}, which gives a structural characterization of the families in $[n]^k$ for which the variable $\zeta_\ff$ is nearly-constant. The third one is the method of spread approximations. These ingredients are put together to actually prove the theorem in Section~\ref{sec4}.

Extremal set theoretic problems often exhibit very different behavior depending on the regime of the parameters. A common phenomenon is that extremal results are  significantly easier to derive provided that $n$ is large enough as a function of $s,k.$ In the case of Theorem~1 and 2 we indeed found a simpler argument in the regime $n\ge k^2s^2$, see Section~\ref{sec-easy}.

We also study the case $k=2$ separately. Aharoni and Howard confirmed their conjecture for $k=2$ in the original paper \cite{AH}. Here, we give another proof for prime $n$ based on the Combinatorial Nullstellensatz. Besides the uniform sequence $f_1 = \ldots = f_s = (s-1)n,$ it gives many other sequences that correspond to non-zero coefficients of a certain polynomial. It appears that many of these sequences are far from being the best possible, but a more thorough investigation of this matter would be very interesting. We shall prove the following.

\begin{thm}\label{thmmain4}
    Let $n = p$ be a prime and $k = 2$. Let $f_1, ..., f_s$ be such that the coefficient of $x_1^{f_1},..., x_s^{f_s}$ in $\prod_{1\leq i<j \leq s} (x_j - x_i)^{2}$ is non-zero modulo $p$. Then $p f_1, ..., p f_s$ is a satisfying sequence. 
\end{thm}
The proof of this theorem is given in Section~\ref{sec6}.

{\bf Note:} P. Saengrungkongka \cite{Saen} has recently made further progress on some of the questions that we study in this paper. Namely, the author proves that $f_i = (i-1)n^{k-1} + c$ is satisfying for $c = \Omega_k(\max\{s^2 n^{k-2}, sn^{k-\frac{3}{2}}\sqrt{\log{s}}\})$, which gives an improvement over our Theorem~\ref{thmmain3} for $n^{\frac{1}{2}+\varepsilon}\ll_k s\ll_k n$. He also gives a generalization of Theorem~\ref{thmmain4} with $n$ not necessarily prime and a more explicit condition on the sequence $f_i$.

\section{Examples of sequences that are not satisfying.}\label{sec2}

The next two claims show that the bound on $n$ in Theorem~\ref{thmmain1} is sharp up to a multiplicative constant.

\begin{cla}
The sequence $(s-1)(n^{k-1}-(n-1)^{k-1})-1, (s-1)n^{k-1},\ldots, (s-1)n^{k-1}$ is not satisfying.     
\end{cla}
\begin{proof}
Consider the following sequence of families (see Fig.~\ref{fig:exapmle1}). Take $F\in \{s\}\times[n]^{k-1}$ and let 
\begin{align*}\F_1 &= \{X\in [s-1]\times[n]^{k-1} : X\cap F\neq \emptyset\},\\
\F_2 = \ldots = \F_s &= [s-1]\times[n]^{k-1} \cup \{F\}
\end{align*}
\begin{figure}[H]
\begin{tikzpicture}
    \foreach \x in {0,2,4,9}{
    \filldraw[color=black, fill=black!3] (\x,0) rectangle (\x+1, 5);};
    \filldraw[color=red!10, opacity=0.5] (0.5, 2.12)rectangle(9.5, 2-0.12);

    \foreach \x in {0,2,4,9}{
    \foreach \y in {0, 1.5, 2, 3, 3.5, 4}{
    \filldraw[color=black] (\x + 0.5, \y + 0.5) circle (0.07);};
    \foreach \y in {0.5, 0.75, 1,  2.35, 2.5, 2.65}{
    \filldraw[color=black] (\x + 0.5, \y + 0.5) circle (0.02);};
    };

    \foreach \y in {2, 3, 3.5, 4}{
    \draw[thick, color=red] (0.5, \y + 0.5) circle (0.12);};
    \draw[thick, color=red] (0.5, 2.12) arc (90:270:0.12);
    \draw[thick, color=red] (9.5, 2-0.12) arc (-90:90:0.12);
    \draw[thick, color=red] (0.5, 2.12)--(9.5, 2.12);
    \draw[thick, color=red] (0.5, 2-0.12)--(9.5, 2-0.12);

    \foreach \y in {2, 3, 3.5, 4}{
    \foreach \x in {2, 4, 9}{
    \draw[thick, color=blue!50] (0.5, \y+0.5)--(\x+0.5, 2);
    };
    };
    \filldraw[color=white] (6.25,0) rectangle (7.75,5);
    \filldraw[color=white] (5.75,0) rectangle (6.0,5);
    \filldraw[color=white] (5.25,0) rectangle (5.5,5);
    \filldraw[color=white] (8,0) rectangle (8.25,5);
    \filldraw[color=white] (8.5,0) rectangle (8.75,5);
    \filldraw[color=black] (7, 2.25) circle (0.02);
    \filldraw[color=black] (7.40, 2.25) circle (0.02);
    \filldraw[color=black] (6.60, 2.25) circle (0.02);

    \draw[decorate,decoration={brace, amplitude=5pt,mirror}] (-0.2, 4.6)--(-0.2, 2.4);
    \draw (-1, 3.5) node{$s-1$};
    \draw[decorate,decoration={brace, amplitude=5pt}] (10.2, 4.6)--(10.2, 0.4);
    \draw (10.6, 2.5) node{$n$};
    
\end{tikzpicture}
\caption{$\ff_i$ for $i\in [2,s]$ consists of all hyperedges that contain a red set and $\ff_1$ consists of all hyperedges that contain a blue two-element set (drawn as a line segment).}
\label{fig:exapmle1}
\end{figure}

The sequence of sizes is $(s-1)n^{k-1} - (s-1)(n-1)^{k-1}, (s-1)n^{k-1}+1, ..., (s-1)n^{k-1}+1$, and the families are easily seen to be cross-dependent.
\end{proof}
Analyzing the previous example, we get the following counterpart of Theorem~\ref{thmmain1}.
\begin{cla}\label{cl5}
The sequence $\{f_i\}_i,$ where $f_i=\min(s-1, i + C\sqrt{s\log s})n^{k-1},$ is not satisfying for $s>s_0(C)$ and $k-1\ge 3Cn\sqrt{\frac{\log s}{s}}$.     
\end{cla}
\begin{proof}
Take the example from the previous claim and compare the sizes with $f_i$.  Clearly, $|\F_i| > f_i$ for $i \in [2,s]$. Suppose that $k = \alpha n + 1$ for some $\alpha > 0$. Then $|\F_1| = (1 - (1-\frac{1}{n})^{k-1})(s-1)n^{k-1} \geq (1 - e^{-\frac{k-1}{n}})(s-1)n^{k-1} = (1-e^{-\alpha})(s-1)n^{k-1}$, where the inequality follows from $1-x \leq e^{-x}$. Thus, as long as $(1-e^{-\alpha})(s-1)>1 + C\sqrt{s\log s}$, the example from the previous claim shows that the sequence $f_i$ is not satisfying. It is a standard calculation to show that this holds for, say, $\alpha>3C\sqrt{\frac{\log s}{s}}$ and $s>s_0(C).$
\end{proof}

The main result of Section~\ref{sec5} (Theorem~\ref{thmmain3}) states that a sequence $\{f_i\}_i$,  $f_i = (i-1)n^{k-1} + 4(s-1)^2 n^{k-2} + 2^{15}s^3 \log_2^3(sk) n^{k-3}$ is satisfying for $n>2^{5}s\log_2(sk).$
If $n > 2^{15}s \log_2^3(sk)$, this implies that a sequence $\{f_i'\}_i$, $f_i' = (i-1)n^{k-1} + 5s^2 n^{k-2}$, is satisfying. In the next two claims we show the complementary bound.

\begin{cla}
    The sequence $g_1 = \lceil \frac s2\rceil \lfloor \frac s2\rfloor n^{k-2}-1, g_2 = \ldots 
 =g_s = sn^{k-1} - \lceil \frac s2\rceil \lfloor \frac s2\rfloor n^{k-2}-1$ is not satisfying.
\end{cla}
\begin{proof}
Put (see Fig.~\ref{fig:exapmle2})
\begin{align*}
    \F_1 &= [\lceil s/2\rceil ]\times[\lfloor s/2\rfloor]\times[n]^{k-2},\\
    \F_2 =\ldots = \F_s &= [\lceil s/2\rceil]\times[n]^{k-1} \cup [n]\times[\lfloor s/2\rfloor]\times[n]^{k-2}.
\end{align*}
\begin{figure}[H]
\begin{tikzpicture}
    \foreach \x in {0,2,4,9}{
    \filldraw[color=black, fill=black!3] (\x,0) rectangle (\x+1, 5);};
    \foreach \x in {0,2,4,9}{
    \foreach \y in {0, 1.5, 2, 3, 3.5, 4}{
    \filldraw[color=black] (\x + 0.5, \y + 0.5) circle (0.07);};
    \foreach \y in {0.5, 0.75, 1,  2.35, 2.5, 2.65}{
    \filldraw[color=black] (\x + 0.5, \y + 0.5) circle (0.02);};
    };

    \foreach \y in {2, 3, 3.5, 4}{
    \foreach \x in {2, 3, 3.5, 4}{
    \draw[thick, color=red!50] (0.5, \y+0.5)--(2.5, \x+0.5);
    };
    };

    \foreach \x in {0,2}{
    \foreach \y in {2, 3, 3.5, 4}{
    \draw[thick, color=blue] (\x + 0.5, \y + 0.5) circle (0.12);};};
    
    \filldraw[color=black] (7, 2.5) circle (0.04);
    \filldraw[color=black] (7.60, 2.5) circle (0.04);
    \filldraw[color=black] (6.40, 2.5) circle (0.04);

    \draw[decorate,decoration={brace, amplitude=5pt,mirror}] (-0.2, 4.6)--(-0.2, 2.4);
    \draw (-0.6, 3.5) node{$\frac{s}{2}$};
    \draw[decorate,decoration={brace, amplitude=5pt}] (10.2, 4.6)--(10.2, 0.4);
    \draw (10.6, 2.5) node{$n$};
    
\end{tikzpicture}
\caption{$\ff_1$ consists of all hyperedges that contain a red two-element set (drawn as a line segment) and $\ff_i$ for $i\in [2,s]$ consists of all hyperedges that contain a blue vertex. The picture represents the case of even $s$.}
\label{fig:exapmle2}
\end{figure}
Note that each set from $\ff_1$ contains at least $2$ elements from the first $\lceil s/2\rceil$ elements in the first part and the first $\lfloor s/2\rfloor$ elements in the second part, and each set from $\ff_i, i>1$ contains at least $1$. Since there are $\lceil s/2 \rceil + \lfloor s/2\rfloor = s$ elements in total, this implies by the pigeon-hole principle that these families are cross-dependent.  

Their sequence of sizes is $\lceil \frac s2\rceil \lfloor \frac s2\rfloor n^{k-2}, sn^{k-1} - \lceil \frac s2\rceil \lfloor \frac s2\rfloor n^{k-2},...,sn^{k-1} - \lceil \frac s2\rceil \lfloor \frac s2\rfloor n^{k-2}$. 
\end{proof}

\begin{cla}
    The sequence $\{g_i\}_i,$ $g_i=(i-1)n^{k-1} + \lceil \frac s2\rceil \lfloor \frac s2\rfloor n^{k-2} - 1,$ is not satisfying for $n\ge s^2/2$.
\end{cla}
\begin{proof}
Take the example from the previous claim. Then always $|\F_1| > g_1$ and, moreover, for $n \ge \frac{s^2}{2}$ we have $|\F_i| > g_i$ for all $i \in [2,s]$.
\end{proof}
In particular, if $s$ is even and $ n = s^2/2,$ then the sequence of families from the previous example shows that the sequence $\{(i-1/2)n^{k-1}-1\}$ is not satisfying. 

\section{Spread approximations: proof of Theorem~\ref{thmmain2}}\label{sec5}

Throughout this section, we think of our families residing on the ground set  $[k]\times[n]$. Accordingly, an element of the ground set is denoted $(j,a)\in [k]\times [n].$ The $k$-partite property of each family is then expressed as follows: each set from each family contains exactly $1$ element of the form $(j,a)$ for each $j\in[k].$ Let us introduce some notation, a part of which is somewhat non-standard. Given families $\ff,\m S$ and sets $X\subset Y,$ define
\begin{align*}
\ff(X)&:=\{F\setminus X:X\subset F, F\in \ff\},\\
\ff[X]&:= \{F: X\subset F, F\in \ff\}, \\
\ff(X, Y)&:= \{F\setminus X: F\cap Y = X\},\\
\ff(\m S)&:=  \bigcup_{A\in \m S}\ff(A),\\
\ff[\m S]&:=  \bigcup_{A\in \m S}\ff[A].
\end{align*}

The approach in this section is a variant of what is called the {\it spread approximation technique}, introduced by Zakharov and the first author in \cite{KuZa}. 

A high-level idea is to split the family into a structured part, which is essentially a lower-uniformity family that inherits some of the properties of the original family, and a small remainder. 

The key notion to obtain such a decomposition is that of an $r$-spread family. Given real $r>1$, a family $\ff$ is called {\it $r$-spread} if for any $X$ we have $|\ff(X)|\le r^{-|X|}|\ff|.$ We say that $W$ is a {\it $p$-random} subset of $[n]$ if each element of $[n]$ is included in $W$ independently and with probability $p.$

The spread approximation technique is based on the following theorem.
\begin{thm}[The spread lemma, \cite{Alw}, a sharpening due to \cite{Tao} and \cite{Sto}]\label{thmtao}
  If for some $n,k,r\ge 1$ a family $\ff\subset {[n]\choose \le k}$ is $r$-spread and $W$ is a $(\beta\delta)$-random subset of $[n]$, then \begin{equation}\label{eqtao}\Pr[\exists F\in \ff\ :\ F\subset W]\ge 1-\Big(\frac 2{\log_2(r\delta)} \Big)^\beta k.\end{equation}
\end{thm}
Actually, Stoeckl~\cite{Sto} showed that we can take $1 + h_2(\delta)$ as the constant in the numerator, where $h_2(\delta) = - \delta \log_2 \delta - (1 - \delta) \log_2 (1 - \delta)$ is the binary entropy of $\delta$.

The main result of this section is as follows.
\begin{thm}\label{thmmain3}
    For $n>2^5 s\log_2(sk)$ the sequence $$f_i = (i-1)n^{k-1} + 4(s-1)^2 n^{k-2} + 2^{15}s^3 \log_2^3(sk) n^{k-3}$$ is satisfying. 
\end{thm}

Theorem~\ref{thmmain2} is an easy corollary: for $n \geq \max\{2^{8}s^{3/2} \log_2^{3/2}(sk), 8 s^2\}$ we have $(i-1)n^{k-1} + 4(s-1)^2 n^{k-2} + 2^{15}s^3 \log_2^3(sk) n^{k-3} \leq i\cdot n^{k-1}$ for all $i\in [s]$, and thus the latter sequence is satisfying as well.

We now turn to the proof of Theorem~\ref{thmmain3}.

Take cross-dependent families $\ff_1,\ldots, \ff_s\subset [n]^k$. Put $r= 2^{5} s\log_2(sk)$. In what follows, it will be convenient for us to denote by $\m A$ the family  $[n]^k$.

The first step of the proof is to construct certain collections $\m S_i$ of subsets of $[k]\times [n]$ called spread approximations of $\ff_i$.  Actually, $\m S_i$ will only contain sets of size at most $2.$ We describe the procedure below. For each $i\in [s]$ we initialize by setting $\F_i' = \F_i$, $\m S_i = \emptyset$. Then, we repeat the following steps:
\begin{itemize}
    \item Choose an inclusion-maximal set $S \subset [k]\times[n]$ such that $|\F_i'(S)| \geq r^{-|S|}|\F_i'|.$ Such exists since $|\F_i'(\emptyset)| \geq r^{-0}|\F_i'|$, and so the empty set satisfies the condition.
    \item If $|S| > 2$ or $\F_i' = \emptyset$, then stop.
    \item Otherwise add $S$ to $\Sf_i$ and redefine $\F_i' := \F_i'\backslash \F_i'[S]$.
\end{itemize}
Once the procedure stops, it outputs the families $\ff'_i$ and $\m S_i,$ where the latter contains only sets of size $\le 2.$
We summarize the properties of these families in the following lemma.
\begin{lem} For each cross-dependent sequence of families $\ff_i$ and $r= 2^{5} s\log_2(sk)$, if $r \leq n$, there exist sequences of families $\Sf_i \subseteq 2^{[k]\times [n]}$ and $\ff_i' \subseteq \ff_i$ such that:
\begin{enumerate}
    \item $\F_i \subset \bigcup_{S\in \Sf_i}\m A[S] \cup \F_i'$; 
    \item $|\F_i'|\leq r^3 n^{k-3}$;
    \item for each $S \in \Sf_i$ there is a subfamily $\F^S_i \subset \F_i$ of sets containing $S$ such that $\F^S_i(S)$ is $r$-spread;
    \item $\Sf_i$'s are cross-dependent.
\end{enumerate}
\end{lem}
\begin{proof}
Take $\Sf_i$ and $\F_i'$ constructed as above for given $\ff_i$ and $r$. 

The first property is immediate from the construction: indeed, $\ff_i\subset \m A$, and at each step we remove those sets from $\ff_i$ that contain $S$. What is left at the end is $\ff_i'$. 
The second property is because, if $\ff_i'$ is nonempty at the last step, then $|S|>2,$ and so
$$|\ff'_i|\le r^{|S|}|\ff'_i(S)|\le r^{|S|}n^{k-|S|}\le r^3n^{k-3}.$$
The last inequality is because $n\ge r$. To verify the third property, note that we can actually take as $\ff_i^S$ the subfamily of all sets containing $S$ at the corresponding step of procedure (i.e., when $S$ was selected as the inclusion-maximal set that violates $r$-spreadness). It is easy to check that inclusion-maximality guarantees that $\ff_i^S(S)$ is $r$-spread: otherwise, find a set $X$ that violate $r$-spreadness of $\ff_i^S(S)$ and replace $S$ with $S\cup X$ in the corresponding step of the procedure. We refer the reader to \cite{KuZa} for more details. 

It remains to check that $\m S_i$ are cross-dependent. Again, this is done in the same way as in \cite{KuZa}, so we present the argument here in a compact form.

Assume the contrary: there is a cross-matching $B_1\in \m S_1, ..., B_s\in \m S_s$. Take the families $\ff_i^{B_i}$ guaranteed by part 3 and note that the families $\m G_i:=\ff_i^{B_i}(B_i, \cup_{j\in[s]}B_j)$ are  $>r/2$-spread. Indeed, this follows from the fact that 
\begin{align*}|\m G_i| = |\ff_i^{B_i}(B_i, \cup_{j\in[s]}B_j)|&\ge |\ff_i^{B_i}(B_i)|-\sum_{\ell\in \cup_{j\in[s]}B_j\setminus B_i}|\ff_i^{B_i}(B_i\cup\{\ell\})|\\
&\ge \left(1-\frac{|\cup_{j\in[s]}B_j|}r\right)|\ff_i^{B_i}(B_i)|> \frac 12 |\ff_i^{B_i}(B_i)|.\end{align*}

The second inequality above follows from property 3, i.e. $\F^{B_i}(B_i)$ being $r$-spread: for an $\ell \notin B_i$ we have $|\F^{B_i}(B_i\cup\{\ell\})| = |\F^{B_i}(B_i)(\{\ell\})| \leq r^{-|\{\ell\}|}|\F^{B_i}(B_i)| = \frac{|\F^{B_i}(B_i)|}{r}$.

Then for any non-empty $X$ disjoint with $B_i$ we have $$|\m G_i(X)|\le |\ff_i^{B_i}(B_i\cup X)|\le r^{-|X|}|\ff_i^{B_i}(B_i)|< 2r^{-|X|}|\m G_i|\le  (r/2)^{-|X|}|\m G_i|,$$
where for the second inequality we once again use part 3.

Color each element of $[k]\times [n] \setminus \cup_{i\in[s]}B_i$ in one of $s$ colours uniformly and independently. Then the set of elements in each colour is a $1/s$-random set. By Theorem~\ref{thmtao} and the choice of $r$, the family $\m G_i$ contains a set colored entirely in colour $i$ with probability strictly greater than $1 - \frac{1}{s}$.\footnote{For its application, we choose $\beta = \log_2(sk)$, $\delta = 1/s\beta$. This way, $r/2= 2^4s\log_2(sk)$, and since our families are $>r/2$-spread, we get something strictly greater than $4$ in the denominator in \eqref{eqtao}.} Thus, by the union bound, there is a coloring such that the family $\m G_i$ contains a set $S_i'$ colored entirely in colour $i$ for all $i \in [s]$. Then sets $B_i \sqcup S_i'$ constitute a cross-matching in $\F_1,\ldots, \F_s$, a contradiction.
\end{proof}

Now we shall change the constructed spread approximation in order to gain a new important property while maintaining the crucial ones among already existing and obtain the following lemma.

\begin{lem}\label{lem-spred-modif}
    For each cross-dependent sequence of families $\ff_i$ and $r= 2^{5} s\log_2(sk)$, if $r \leq n$, there exist sequences of families $\Sf_i \subseteq 2^{[k]\times [n]}$ and $\ff_i' \subseteq \ff_i$ with the following properties:
    \begin{itemize}
        \item $\F_i \subset \bigcup_{S\in \Sf_i}\m A[S] \cup \F_i'$; 
        \item $|\F_i'|\leq r^3 n^{k-3}$;
        \item $\Sf_i$'s are cross-dependent;
        \item for each $S \in \m S_i$, we have $|S| \leq 2$;
        \item for $\m S_i^{(2)} := \{S \in \Sf_i\, :\, |S| = 2\}$ we have $|\m S_i^{(2)}| \leq 4(s-1)^2$.
    \end{itemize}
\end{lem}
\begin{proof}
Take $\m S_i$ and $\ff_i'$ from the previous lemma. We will modify $\m S_i$, but only so that properties 1 and 4 are preserved.  Put $\Sf_i^{(\ell)} := \{S\in \Sf_i : |S| = \ell\}$.  

If an element $(j,a)\in [k]\times[n]$ lies in at least $2s-1$ sets from $\Sf_i^{(2)}$ then add $\{(j,a)\}$ to $\Sf_i^{(1)}$ and delete all sets containing $(j,a)$ from $\Sf_i^{(2)}$. First, let us check that $\m S_1,\ldots, \m S_s$ remain cross-dependent.  Arguing indirectly, assume that $\Sf_1,...,\Sf_i\cup\{\{(j,a)\}\},...,\Sf_s$ have a cross-matching $B_1,...,B_s$. First,  $B_i = \{(j,a)\}$, otherwise there was a cross-matching in the original families $\m S_1,\ldots, \m S_s$. Therefore, none of sets $B_1,...,B_{i-1},B_{i+1},...,B_s$ contain $(j,a)$. We have $|\cup_{j\in[s]\setminus \{i\}}B_j|\le 2s-2$, and thus there is a set $B_i'$ in $\m S_i^{(2)}$ that contains $(j,a)$ and avoids   $\cup_{j\in[s]\setminus \{i\}}B_j$. Replacing $B_i$ by $B_i'$ gives a cross-matching in $\m S_1\ldots, \m S_s$, a contradiction.

At the same time, $\m A[\m S_i]$ can only grow when replacing sets containing $\{(j,a)\}$ by $\{(j,a)\}$ itself. 
Also note that if $\{(j,a)\}$ was already in $\Sf_i$, deleting all $2$-element sets containing $(j,a)$ has no effect on $\m A[\Sf_i^{(1)} \cup \Sf_i^{(2)}]$. Thus we can w.l.o.g. assume that each element $(j,a)\in [k]\times[n]$ is contained at most $2(s-1)$ sets from $\Sf_i^{(2)}$. 

Now suppose that $|\Sf_i^{(2)}| > 4(s-1)^2$. Then we can replace  $\m S_i$  with $\{\emptyset\}$, and the new families will remain cross-dependent. Indeed, assume that there is a cross-matching in $B_1\in \m S_1,\ldots, B_{i-1}\in \m S_{i-1}, B_{i+1}\in \m S_{i+1},\ldots, B_s\in \m S_s$.  We have $|\cup_{j\in [s]\setminus \{i\}}B_i|\le 2s-2$. Due to the  assumption of the previous paragraph, there are at most $4(s-1)^2$ sets in $\Sf_i^{(2)}$ that may contain one of the elements of $\cup_{j\in [s]\setminus \{i\}}B_i$. Therefore, there is at least one set from $\m S_{i}^{(2)}$ disjoint from all the sets in the cross-matching, a contradiction. We conclude that the families $\m S_1,\ldots,\m S_{i-1}, \m S_{i+1},\ldots, \m S_s$ are cross-dependent, and that we can replace $\m S_i$ by $\{\emptyset\}$ without violating the cross-dependent property. Consequently, we can obtain $|\Sf_i^{(2)}| \leq 4(s-1)^2$ for each $i$. 
\end{proof}

\begin{proof}[Proof of Theorem~\ref{thmmain3}]

Suppose that the families $\F_i$ are cross-dependent. Then we can apply Lemma~\ref{lem-spred-modif} and get the spread approximations $\Sf_i$. Note that either $\Sf_i$ contains an empty set or $|\F_i| \leq \big|\m A\big[\Sf_i^{(1)} \cup \Sf_i^{(2)}\big]\big| + |\F_i'|$.
Since spread approximations are cross-dependent, for at least one $i$ we must have $\{\emptyset\} \notin \Sf_i$ and $|\Sf_i^{(1)}| \leq i-1$. Indeed, suppose that each $\Sf_i$ contains either the empty set of $\geq i$ one-element sets. Then we can construct a cross-matching $\mathcal M$ for $\Sf_i$'s as follows: start with $\m M = \emptyset$ and for each $i = 1, ..., s$: if $\emptyset \in \Sf_i$ add $\emptyset$ to $\m M$, otherwise add to $\m M$ a one-element set which is in $\Sf_i$ but not already in $\M$ (such exists since in this case $|\Sf_i^{(1)}| \geq i$ while $|\M| = i-1$ before the $i$-th step). Since distinct one-element sets are disjoint and the empty set is disjoint from everything, $\M$ is a cross-matching, which contradicts $\Sf_i$ being cross-dependent.

Combining it with the bound on $|\Sf_i^{(2)}|$ and $\ff_i'$, we get 
$$|\F_i| \leq (i-1)n^{k-1} + 4(s-1)^2 n^{k-2} + r^3 n^{k-3} = (i-1)n^{k-1} + 4(s-1)^2 n^{k-2} + 2^{15}s^3 \log_2^3(sk) n^{k-3}.$$ 
On the contrapositive, if for each $i\in[s]$ the size of $\ff_i$ is strictly bigger than the right-hand side, then they contain a cross-matching.
\end{proof}

\section{Anticoncentration}\label{sec3}

Fix a family $\ff\subset [n]^k$. One of the important steps in our approach is to study the random variable $$\xi_\ff(\mathcal M) = |\ff\cap \mathcal M|,$$
where $\mathcal M$ is a uniformly random perfect matching in $[n]^k,$ i.e., a collection of $n$ pairwise disjoint $k$-element sets $M_1,\ldots, M_n$ that together cover the ground set $[n]\times [k]$.

The following anticoncentration result gives a structural characterization of the families $\ff$ for which $\xi_\ff$ is almost constant. Informally, it states that such an $\ff$ shoud be in some sense close to a natural example of the family for with $\xi_\ff(\mathcal M)$ is precisely a constant, namely $\ff = [s-1]\times [n]^{k-1}$ (or its image under an isomorphism of $[n]^{k}$).

To state and prove the theorem we shall think of ${\bf x}\in [n]^k$ as of a vector $(x_1,\ldots, x_k),$ where $x_i\in [n].$ For $a \in [n]$ and $i \in [k]$ we call a family  $$\m H_i(a) := \{{\bf x} =(x_1,\ldots, x_k)\in [n]^k : x_i = a\}$$  a {\it hyperplane}.

\begin{thm}\label{antic}
     Let $\mathcal{F} \subset [n]^k$, $|\mathcal{F}| = (s-1)n^{k-1} + 1$ and define 
    $$p_\ff:=\mathbb{P}[\xi_\ff\neq s-1].$$
    \begin{enumerate} \item Let $n\ge 4$ and $p_\ff<1/8$. Fix some $i\in[k]$. Then  we have two possibilities. The first possibility is that for any 
    $a\in [n]$ we have
    $$|\m H_i(a) \cap \mathcal{F}| \leq 2p_\ff n^{k-1} + (s-1) n^{k-2}.$$
    The second possibility is that there exists a partition $[n] = F\sqcup T,$ such that for any $a\in F$ we have
    $$|\m H_i(a) \cap \mathcal{F}| \geq (1 - 2p_\ff)n^{k-1},$$
    in which case we call $\m H_i(a)$ {\em $\ff$-fat}, and for any $a\in T$ we have
    $$|\m H_i(a) \cap \mathcal{F}| \leq 2p_\ff n^{k-1},$$
    in which case we call $\m H_i(a)$ {\em $\ff$-thin}.
    \item      Suppose that $n \ge 8$, $s \ge 3$ and 
    $p_\ff\leq \frac{1}{4n(s-1)}.$
    Then {\it $\ff$-fat} hyperplanes are {\em parallel}, i.e., they all have the  form $\m H_i(a)$ for some fixed $i\in[k]$.
    \item      Suppose that $n \ge 8$, $s \ge 3$ and 
    $p_\ff\leq \min\{\frac{1}{4n(s-1)}, \frac{1}{4kn}\}.$
    Then there are exactly $s-1$ $\ff$-fat hyperplanes.
\end{enumerate} 
\end{thm}

\begin{proof} 1. 
    W.l.o.g we can assume $i = 1$.

Let \begin{equation}\label{eqpab}\m P_{a,b} := \{{\bf s} \in [n]^{k-1}: (a,{\bf s}) \in \F, (b,{\bf s}) \notin \F\}.\end{equation} Fix some $a\neq b\in [n]$. Partition all perfect matchings on $[n]^k$ into pairs: 
{\small $$\m M_1:=\big\{(1,{\bf s_1}), ...,(a,{\bf s_a}),...,(b,{\bf s_b}),...,(n,{\bf s_n})\big\} \longleftrightarrow \m M_2:=\big\{(1,{\bf s_1}), ...,(a,{\bf s_b}),...,(b,{\bf s_a}),...,(n,{\bf s_n})\big\}.$$}
Select only the pairs $\m M_1 \longleftrightarrow \m M_2$ such that ${\bf s_a} \in \m P_{a,b}$ and ${\bf s_b} \notin \m P_{a,b}$. There are at least  $|\m P_{a,b}|$ choices for ${\bf s_a}$ and $n^{k-1} - |\m P_{a,b}|$ choices for ${\bf s_b}$. Some of them, however, are not disjoint. We deal with it in the following claim.
\begin{cla}The number of disjoint pairs between $\m P_{a,b}$ and $\bar{\m P}_{a,b}:=[n]^{k-1}\setminus \m P_{a,b}$ is at least $$\frac{(n-1)^{k-1}-(n-1)^{k-3}}{n^{k-1}}|\m P_{a,b}|(n^{k-1} - |P_{a,b}|).$$
\end{cla}
\begin{proof}
    We use a variant of the expander-mixing lemma \cite[Theorem 9.2.1]{AS}, which states the following. Take a $d$-regular graph $G$ with an adjacency matrix $A$ that has eigenvalues $d=\lambda_1\ge \lambda_2\ge\ldots.$ Then the number of edges between a set $X\subset X(G)$ and its complement $\bar X$ is at least $\frac {(d-\lambda_2)|X||\bar X|}{|V(G)|}$. We apply this statement to the $(k-1)$st power of the complete graph $K_n$. Its vertex set is $[n]^{k-1}$, and edges connect disjoint sets. The eigenvalues of this graph are products of the eigenvalues of $K_{n}$: $(-1)^i (n-1)^{k-i}.$ That is, in terms of the expander-mixing lemma, $d = (n-1)^{k-1}$ and $\lambda_2 = (n-1)^{k-3}$. Plugging these values, we get the statement. 
\end{proof}

Using this claim, we get that there are at least $$\frac{(n-1)^{k-1}-(n-1)^{k-3}}{n^{k-1}}|\m P_{a,b}|(n^{k-1} - |P_{a,b}|) ((n-2)!)^{k-1}$$ disjoint pairs of matchings as before the claim. At the same time, $\xi_\ff(\m M_1)\ne \xi_\ff(\m M_2)$ and thus $\xi_\ff(\m M_i)\ne s-1$ for at least one $i\in[2]$. Indeed, since ${\bf s_a} \in \m P_{a,b}$, we have $(a, {\bf s_b})\in \m F$ and $(b, {\bf s_a}) \notin \m F$. If $(b, {\bf s_b})\in \m F$ then $\F\cap \{(a,{\bf s_a}), (b, {\bf s_b})\} = 2 > 1 \geq \F\cap \{(a,{\bf s_b}), (b, {\bf s_a})\}$. In the converse case $(b, {\bf s_b})\notin \m F$ we necessarily have $(a, {\bf s_b}) \notin \m F$ (otherwise ${\bf s_b} \in \m P_{a,b}$), and thus $\F\cap \{(a,{\bf s_a}), (b, {\bf s_b})\} = 1 > 0 = \F\cap \{(a,{\bf s_b}), (b, {\bf s_a})\}$. Since $\m M_1\setminus \{(a,{\bf s_a}), (b, {\bf s_b})\} = \m M_2 \setminus \{(a,{\bf s_b}), (b, {\bf s_a})\}$, this gives  $\xi_\ff(\m M_1) > \xi_\ff(\m M_2)$ in any case. It implies that the number of such pairs is at most $p_\ff (n!)^{k-1}$: since each matching participates in at most one pair and at least one matching in a pair satisfies $\xi_\ff(\m M_i)\ne s-1$, the number of pairs is at most the number of such matchings. Hence,

$$\frac{(n-1)^{k-1}-(n-1)^{k-3}}{n^{k-1}}|\m P_{a,b}|\big(n^{k-1} - |\m P_{a,b}|\big) \big((n-2)!\big)^{k-1} \leq p_\ff (n!)^{k-1},$$
which is equivalent to
\begin{equation}\label{eqcompare}\big(1-(n-1)^{-2}\big)|\m P_{a,b}|\big(n^{k-1} - |\m P_{a,b}|\big) \leq   p_\ff n^{2k-2}.\end{equation}
Looking at this quadratic inequality, we see that (for $n-1\ge 3,$ $p_\ff\le 1/5$) at least one of the following should hold:
\begin{align}\label{eqslim}|\m P_{a,b}| &< 2p_\ff n^{k-1},\\
\label{eqfat}|\m P_{a,b}| &> (1 - 2p_\ff)n^{k-1},\end{align}
otherwise the product on the LHS of \eqref{eqcompare} is at least $(1-1/9)(2p_\ff)(1-2p_\ff)n^{2k-2}>p_\ff n^{2k-2}.$ 

Next, we consider two cases, corresponding to the two cases in the statement. Assume first that for any $a,b\in[n]$ \eqref{eqslim} holds. Then $$|\m H_1(b)\cap \F| \geq |\m H_1(a)\cap \F| - |\m P_{a,b}| > |\m H_1(a)\cap \F|-2p_\ff n^{k-1}.$$ Summing over $b\in [n]$, we get $$(s-1)n^{k-1} +1\geq |\F| > n|\, \m H_1(a)\cap \F|-2(n-1)p_\ff n^{k-1}.$$ Doing  routine calculations, we get  $|\m H_1(a)\cap \F| \leq 2p_\ff n^{k-1} + (s-1)n^{k-2}$, valid for each $a\in[n]$. This corresponds to the first case.

Next, assume that there are $a,b\in[n]$ such that $|\m P_{a,b}| > (1 - 2p_\ff)n^{k-1}$. Then  $|\m H_1(a)\cap \F| \geq |\m P_{a,b}|\ge (1 - 2p_\ff)n^{k-1},$ where the first inequality holds for  any $b \neq a$. Likewise, $|\m H_1(b)\cap \F| \leq n^{k-1}-|\m P_{a,b}|\le  2p_\ff n^{k-1}.$  We claim that in this case for any $c\in[n]$ we either have $|\m H_1(c)\cap \F|\le 2p_\ff n^{k-1}$ or $|\m H_1(c)\cap \F|\ge (1-2p_\ff) n^{k-1}$. Arguing indirectly, assume that $ (1-2p_\ff) n^{k-1}>|\m H_1(c)\cap \F|\ge\frac 12 n^{k-1}$. Then 
$$(1-2p_\ff) n^{k-1}>|\m H_1(c)\cap \F|\ge |\m P_{c,b}|\ge |\m H_1(c)\cap \F|-|\m H_1(b)\cap \F|\ge \big(\frac 12 - 2p_\ff\big)n^{k-1}>2p_\ff n^{k-1},$$
a contradiction with \eqref{eqslim}, \eqref{eqfat}. In the case $ \frac 12 n^{k-1}\ge|\m H_1(c)\cap \F|> 2p_\ff n^{k-1}$ we have
\begin{align*}
    (1-2p_\ff) n^{k-1} &> n^{k-1} - |\m H_1(c)\cap \F| \geq |\m P_{a,c}| \\
    &\geq |\m H_1(a)\cap \F| - |\m H_1(c)\cap \F| \geq (1-2p_\F)n^{k-1} - \frac{1}{2}n^{k-1} > 2p_\F n^{k-1},
\end{align*}
again a contradiction with \eqref{eqslim}, \eqref{eqfat}.

This situation corresponds to the second case.
\vskip+0.1cm
2. Suppose that the hyperplane $\m H_i(a)$ is $\m F$-fat, and also that there is at least one $\m F$-fat hyperplane of the form $\m H_j(b)$ for $j \neq i$. There are at most $s-1$ values $b$ such that $\m H_j(b)$ is $\m F$-fat, because otherwise, using that parallel hyperplanes are disjoint, we have  $$|\ff|\ge s \cdot (1-2p_\ff)n^{k-1}\ge\Big(1-\frac 1{2n}\Big)s n^{k-1}>(s-1)n^{k-1}+1.$$
Also, by the previous part, for any $b$ such that $\m H_j(b)$ is not $\m F$-fat, it is $\m F$-thin, that is, we have $|\m H_j(b)\cap \ff| \leq 2p_\ff n^{k-1}$. Assume w.l.o.g. that $H_j(1), H_j(2), ..., H_j(\ell)$ ($1\leq \ell\leq s-1$) are $\m F$-fat and $H_j(\ell + 1), H_j(\ell + 2), ..., H_j(n)$ are $\m F$-thin. Then

\begin{align*}
    \Big(1 - \frac{1}{2(s-1)n}\Big)n^{k-1} &\leq (1-2p_\ff)n^{k-1} \leq |\m H_i(a) \cap \ff| =\\
    &= \sum_{b = 1}^{n}|\m H_i(a) \cap \ff \cap \m H_j(b)| \leq \sum_{b=1}^\ell |\m H_i(a) \cap \m H_j(b)| + \sum_{b = \ell+1}^n |\m H_j(b)\cap \ff| \\
    &\leq  (s-1)n^{k-2} + n\cdot 2p_\ff n^{k-1} \leq (s-1)n^{k-2} + \frac{n^{k-1}}{2(s-1)},
\end{align*}
and thus $$(s-1)n^{k-2} + \frac{n^{k-1} + n^{k-2}}{2(s-1)} \geq n^{k-1}.$$

Since $1 \leq s-1 \leq n-1$ and the function $x n^{k-2} + \frac{n^{k-1} + n^{k-2}}{2x}$ is convex on $\R^+$, we have $$(s-1)n^{k-2} + \frac{n^{k-1} + n^{k-2}}{2(s-1)} \leq \max\Big\{n^{k-2} + \frac{n^{k-1} + n^{k-2}}{2}, (n-1)n^{k-2} + \frac{n^{k-1} + n^{k-2}}{2(n-1)}\Big\}.$$

For $n\geq 8$, $n^{k-2} + \frac{n^{k-1} + n^{k-2}}{2} \leq \frac{11}{16}n^{k-1} < n^{k-1}$ and $(n-1)n^{k-2} + \frac{n^{k-1} + n^{k-2}}{2(n-1)} = n^{k-1} - n^{k-2}\big(\frac{1}{2} - \frac{1}{n-1}\big) < n^{k-1}$, a contradiction.

\vskip+0.1cm
3. We shall first prove that $\ff$-fat hyperplanes exist. Suppose that none of the hyperplanes of the form $\m H_1(a)$ is $\ff$-fat. Recall the definition \eqref{eqpab} of $\m P_{a,b}$ from the proof of the first point. Since $|\m P_{a,b}| \leq |\m H_1(a)\cap \ff| < (1-2p_\ff)n^{k-1}$ by assumption, we see that option \eqref{eqfat} cannot hold, so for every $1\leq a,b \leq n$ we have \eqref{eqslim}, i.e. $|\m P_{a,b}| < 2p_\ff n^{k-1}$. 

From now on we shall treat $[n]^k$ as $\Z_n^k$ using standard isomorphism. Note that $|\m P_{a,b}| + |\m P_{b,a}| = |(\ff\Delta(\ff+(a-b)e_1))\cap\m H_1(a)|$. Let $\chi_\ff$ be the characteristic function of $\ff$, and consider its Fourier decomposition $\chi_\ff({\bf x}) = \sum_{{\bf y}\in \Z^k_n}\hat{\chi}_\ff({\bf y})\omega^{{\bf x}\cdot {\bf y}}$, where $\omega$ is a primitive $n$-th root of unity. 
Define characteristic functions of the shifts of $\ff$: $\chi_{\ff + {\bf z}}({\bf x}) \equiv \chi_\ff({\bf x} - {\bf z})$. We have $\hat\chi_\ff({\bf x}) = \frac 1{n^k}\sum_{y\in \mathbb Z_n^k} \chi_\ff({\bf y})\omega^{-{\bf x}\cdot {\bf y}}$ and 
$$\hat\chi_{\ff+{\bf z}}({\bf x}) = \frac 1{n^k}\sum_{{\bf y}\in \mathbb Z_n^k} \chi_\ff({\bf y-z})\omega^{-{\bf x}\cdot {\bf y}} = \frac 1{n^k}\sum_{{\bf y}\in \mathbb Z_n^k} \chi_\ff({\bf y})\omega^{-{\bf x}\cdot {\bf (y+z)}} = \omega^{-{\bf x}\cdot {\bf z}}\hat\chi_\ff({\bf x}).$$
We also shall use that for any $y_1\in \mathbb Z_n\setminus \{0\}$ we have $$\sum_{\ell = 0}^{n-1} |1 - \omega^{-\ell y_1}|^2=\sum_{\ell = 0}^{n-1} (1 - \omega^{-\ell y_1})(1 - \omega^{\ell y_1})=2n-\sum_{\ell = 0}^{n-1} \omega^{-\ell y_1}-\sum_{\ell = 0}^{n-1} \omega^{\ell y_1}=2n.$$
Overall, we can write
\begin{align*}
    n^2\cdot 4p_\mathcal{\F}n^{k-1} &> \sum_{1\leq a,b \leq n} (|\m P_{a,b}| + |\m P_{b,a}|) = \sum_{\ell = 0}^{n-1}\sum_{a = 1}^{n}(|\m P_{a,a-\ell}| + |\m P_{a-\ell,a}|)
    \\&= \sum_{\ell = 0}^{n-1}\sum_{a = 1}^{n} |(\ff\Delta(\ff+\ell e_1))\cap\m H_1(a)|\\
    &= \sum_{\ell = 0}^{n-1}|\ff\Delta(\ff+\ell e_1)| = \sum_{\ell = 0}^{n-1}||\chi_{\ff} - \chi_{\ff + \ell e_1}||_2^2 \\
    &= n^k\sum_{\ell = 0}^{n-1}\sum_{{\bf y}\in \Z^k_n}|\hat \chi_{\ff}({\bf y})|^2 |1 - \omega^{-\ell y_1}|^2 = n^ {k}\cdot 2n\sum_{{\bf y} : y_1 \neq 0} |\hat \chi_{\ff}({\bf y})|^2.
\end{align*}

Analogously,  if there are no $\ff$-fat hyperplanes among $\m H_i(a)$ for  $i\in [k]$, then we get 

$$\sum_{{\bf y} : y_i \neq 0} |\hat \chi_{\ff}({\bf y})|^2 < 2p_\ff.$$

If there are no $\ff$-fat hyperplanes at all, then 

$$\frac{|\ff|}{n^k} - \bigg(\frac{|\ff|}{n^k}\bigg)^2 = \sum_{{\bf y} \neq 0}|\hat \chi_{\ff}({\bf y})|^2 \leq \sum_{{\bf y}}|\{i : y_i \neq 0\}|\cdot|\hat \chi_{\ff}({\bf y})|^2 < 2kp_\ff.$$

Since $2\leq s \leq n$, we have $\frac{1}{n}\leq \frac{|\ff|}{n^k} \leq 1 - \frac{2}{3}\cdot\frac{1}{n}$, thus $\frac{|\ff|}{n^k} - \big(\frac{|\ff|}{n^k}\big)^2 \geq \frac{1}{2n}$, and so $p_\ff > \frac{1}{4kn}$, which contradicts the assumption. Consequently, we do have $\ff$-fat hyperplanes. 

Since $p_\ff \leq \frac{1}{4n(s-1)}$, all $\ff$-fat  hyperplanes are parallel by part 2 of the theorem.    Suppose w.l.o.g. that $\ff$-fat hyperplanes  have the form $\m H_1(a)$, where $a\in F\subset [n].$ The hyperplanes $\m H_1(b),$ $b\in T:=[n]\setminus F$ are $\ff$-thin. 

Suppose that $|F|<s-1$. Then $$|\ff|\le |F|\cdot n^{k-1}+|T|\cdot 2p_\ff n^{k-1}\le (s-2)n^{k-1}+ 2p_\ff n^{k}<(s-1)n^{k-1},$$ provided $p_\ff<\frac 1{2n}$. 

Suppose that $|F|>s-1$. then $$|\ff|\ge |F|\cdot (1-2p_\ff)n^{k-1}\ge\big(1-\frac 1{2n})s n^{k-1}>(s-1)n^{k-1}+1,$$
where the second inequality is valid provided $p_\ff\le \frac 1{4n}.$
\end{proof}

\section{(Anti)concentration: proof of Theorem~\ref{thmmain1}}\label{sec4}




In this section we shall prove Theorem~\ref{thmmain1}.

For the proof we will need the following concentration result proven in \cite{KKconc}:

\begin{thm}[Kiselev, Kupavskii \cite{KKconc}]\label{conc}
    Let $\F \subset [n]^k$, $|\F| = \alpha n^k$. Let $\eta$ be a random variable $|\F\cap \M|$, where $\M$ is a uniformly random perfect matching in $[n]^k$. Then for any $\lambda > 0$ and $\delta \in \{-1,1\}$ the following holds

    $$\mathbb{P}[\delta\cdot(\eta - \alpha n) \geq 2\lambda] \leq 2\,\exp \big(-\frac{\lambda^2}{\alpha n/2 + 2\lambda}\big).$$
\end{thm}

\begin{proof}[Proof of Theorem~\ref{thmmain1}]
Assume the contrary, i.e. that there exists a cross-dependent sequence of families $\{\F_i\}_i$, such that $|\F_i| > f_i$. We may w.l.o.g. assume that $|\F_i| = \lfloor f_i + 1\rfloor$. Put $$M := \left\lceil C\sqrt{s\log s} \right\rceil + 1.$$ Note that in this sequence there are $M$ "large" families, i.e. families of  size $(s-1)n^{k-1} + 1$.

We have three cases. In the first case, $|\bigcup_{i=s-M+1}^s\F_i\cap \M|$ is equal to $s-1$ with high probability. By anticoncentration, this imposes a special structure on all "large" $\F_i$'s. In the second case, we deal with $n \geq s^5$. We use spread approximations and show that if families are cross-dependent, then we must get the same special structure as in the first case. In the remaining case, $|\bigcup_{i=s-M+1}^s\F_i\cap \M|$ is not too concentrated and $n \leq s^5$. We will show that there exists a cross-matching in this case too using methods from \cite{KKconc}.

We first prove that the special structure of the big families that we shall obtain in the first two cases implies the existence of a cross-matching.

\begin{cla}\label{cl:str-to-matching}
    If for a sequence of families $\{\F_i\}_i$ such that $|\F_i| > f_i$ we have that 
    \begin{equation}\label{assconc}
        |\F_i \cap \m H_1(t)| \geq \Big(1-\frac{1}{2ns}\Big)n^{k-1}\text{ for any }i\in [s-M+1,s]\text{ and any }t \in [s-1],
    \end{equation} 
    then $\ff_1, ..., \ff_s$ contain a rainbow matching.
\end{cla}
\begin{proof}
    Condition \eqref{assconc} implies that  \begin{equation}\label{eqinters}\left|\bigcap_{i = s-M+1}^s \F_i \cap \m H_1(t)\right| \geq \left(1-\frac{M}{2ns}\right)n^{k-1}\end{equation} for any $t \in [s-1]$.

Also, we can assume that small families almost entirely lie in $\cup_{t\in [s-1]} \m H_1(t)$. That is, we have the following

\begin{cla}
    If $|\F_j\backslash\cup_{t\in [s-1]} \m H_1(t)| > 4\sqrt{s\log s}\, n^{k-1}$ for some $j\in[s-M]$, then there exists a cross-matching.
\end{cla}

\begin{proof}
    Put $\ff'_j:= \F_j\backslash \cup_{t\in [s-1]} \m H_1(t)$. By the concentration result (Theorem~\ref{conc} with $\alpha = \frac{i + C\sqrt{s\log s}}{n}$, $\delta = -1$ and $\lambda = \frac{C\sqrt{s\log s}}{2}$), for a uniformly random matching $\M$ it holds that $\Pp[|\M\cap \F_i| \leq i] \leq 2s^{-C^2/4}$ for $i\in[s-M]$ and sufficiently large $s$.  Also, assuming that $|\F'_j| > 4\sqrt{s\log s}\, n^{k-1}$, the concentration result (Theorem~\ref{conc} with $\alpha = \frac{4\sqrt{s\log s}}{n}$, $\delta = -1$ and $\lambda = 2\sqrt{s\log s}$) implies that $\Pp[|\M\cap \F_j'| = 0] \leq 2s^{-4}$. For each $t \in [s-1]$ for a (random) set $X_t(\M) := \m H_1(t) \cap \M$ the inequality \eqref{eqinters} implies that $\Pp[X_t(\M) \notin \bigcap_{i=s-M+1}^s \F_i] \leq \frac{M}{2ns}$. Then for $C \geq 4$ and $s>s_0(C)$ by the union bound we conclude that there exists a matching $\M_0$ such that the following holds simultaneously: \begin{itemize}
        \item $|\M_0\cap \F_i| \geq i+1$ for all $i\in [s-M]$,
        \item $\M_0\cap \F_j' \neq \emptyset$, and
        \item $X_t(\M_0) \in \bigcap_{i=s-M+1}^s \F_i$ for all $t\in [s-1]$.
    \end{itemize}    
    Fix such an $\M_0$. Then we can choose a cross-matching in $\ff_1,\ldots, \ff_s$ (actually, in $\ff_1\cap \m M_0,\ldots, \ff_s\cap \m M_0$) as follows: choose an element $F_j$ in $\ff'_j\cap \m M_0$, then sequentially for $i\in[s-M]\backslash\{j\}$ choose some $F_i$ in $\F_i\cap \m M_0$. Since $F_j\not \in \cup_{t\in[s-1]}\m H_1(t)$ at least $s-1-(s-M-1) = M$ sets $X_t(\M_0),$ $t\in [s-1]$, are not yet chosen. Assign them arbitrarily to the "large" families $\ff_{s-M+1},\ldots, \ff_s$.
\end{proof}

Using the last claim, we can w.l.o.g. assume that for each $i \in [s-M]$

\begin{equation}\label{adstr}
    \left|\m F_i\setminus \bigcup_{t\in[s-1]}\mathcal H_1(t)\right|\leq 4\sqrt{s\log s}\, n^{k-1}.
\end{equation}

Since $|\F_s| > (s-1)n^{k-1}$, $\F_s$ contains a set $F_s = (x_1,...,x_k)\not\in\cup_{j\in[s-1]}\m H_1(t)$. Let $\mathfrak M'$ be the collection of all perfect matchings that contains $F_s$, and let $\M'$ be a uniformly random element of $\mathfrak M'$. Then $\M'$ is a uniformly random matching on the cube 
$$\m C' := [n]\backslash \{x_1\}\times ...\times [n]\backslash \{x_k\}.$$ Denote $\ff_i' = \ff_i\cap\m C'$. For $1\leq i \leq s-M$ we have 
$$|\mathcal F_i'|\ge |\mathcal F_i| -\left|\m F_i\setminus \bigcup_{t\in[s-1]}\mathcal H_1(t)\right|-\left|\bigcup_{t\in[s-1]}\mathcal H_1(t)\cap ([n]^k \setminus \m C')\right|.$$
We bound the second term using \eqref{adstr} and  get 
\begin{align*}
    |\ff'_i| &\geq (i + C\sqrt{s\log s})n^{k-1}  -  4\sqrt{s\log s}\, n^{k-1} - k(s-1)n^{k-2}\\ &\geq  \left(i + \frac{C}{2}\sqrt{s\log s}\right)n^{k-1} > \left(i + \frac{C}{2}\sqrt{s\log s}\right)(n-1)^{k-1}.
\end{align*}
In the second inequality we use $k \leq \frac{n(\frac{C}{2}-4)}{\sqrt{s/\log s}}$, which is implied by $k \leq \frac{Cn}{3\sqrt{s/\log s}}$ for $C\ge 20$. Thus, by Theorem~\ref{conc} with $n-1$ instead of $n$, $\alpha = \frac{i + \frac{C}{2}\sqrt{s\log s}}{n-1}$, $\delta = -1$ and $\lambda = \frac{C\sqrt{s\log s}}{4} + \frac{1}{2}$, $\Pp[|\ff'_i\cap \m M'| \leq i-1] \leq s^{-\frac{C^2}{20}}$.

For each hyperplane $\m H_1(t)$, $t \in [s-1]$ we have, by \eqref{eqinters}, $|\bigcap_{i = s-M+1}^s\ff_i'\cap\m H_1(t)| \geq (1-\frac{M}{2ns})n^{k-1} - (n^{k-1} - (n-1)^{k-1}) \geq (n-1)^{k-1}-\frac{M}{2ns}n^{k-1} \geq (n-1)^{k-1}(1 - \frac{2M}{3ns})$, where the last inequality follows from the inequality $n^{k-1}\leq \frac{4}{3}(n-1)^{k-1}$, which is easy to verify for $n>10k$, say. Thus the element of $\M'$ in $\m H_1(t)$ avoids $\bigcap_{i = s-M+1}^s \F'_i$ with probability at most $\frac{2M}{3ns}$. Thus, by the union bound, the probability that $\M'$ intersects some $\F'_i$ with $i\in [s-M]$ in less than $i$ elements or $\bigcap_{i = s-M+1}^s \F'_i$ in less than $s-1$ element is at most $s^{-\frac{C^2}{20}+1} + \frac{2Ms}{3ns} < s^{-\frac{C^2}{20}+1} + \frac{2M}{3n} < 1$ for sufficiently large $s$. Consequently, there is a cross-matching $F_1\in \ff_1,...,F_{s-1}\in \ff_{s-1}$  such that $F_i\cap F_s = \emptyset$ for each $i\in[s-1]$. The sets $F_1,...,F_{s-1}, F_s$ form a full cross-matching, a contradiction.
\end{proof}
\vspace{2mm}

Now we are ready to consider the three cases.

\vspace{1mm}

\textbf{Case 1:} $\Pp[|\bigcup_{i = s-M+1}^s \F_i \cap \M| > s-1] < \frac{1}{4n^3}$.

Our anticoncentration result Theorem~\ref{antic} gives bounds that depend on $p_\ff=\Pp[|\F\cap\M| \neq s-1]$, so in order to use the assumption we need the following simple claim.

\begin{cla}
    For any $\F \subset [n]^{k}$ with $|\F| \geq (s-1)n^{k-1}$ we have $$\Pp[|\M\cap\F| > s-1] \geq \frac{\Pp[|\M\cap\F| \neq s-1]}{n}.$$
\end{cla}
\begin{proof}
We have
    \begin{align*}
    s-1 &\leq \E[|\M\cap\F|] = \sum_{\ell = 0}^{n}\ell\cdot \Pp[|\M\cap\F| = \ell] \\ &\leq  (s-2)\Pp[|\M\cap\F| < s-1] + (s-1)\Pp[|\M\cap\F| = s-1]+ n\Pp[|\M\cap\F| > s-1] \\&= (s-2)\Pp[|\M\cap\F| < s-1] + (s-1)(1 - \Pp[|\M\cap\F| < s-1] \\&\ \ \ \ \ \ - \Pp[|\M\cap\F| > s-1]) + n\Pp[|\M\cap\F| > s-1],
\end{align*}

and thus, substituting $\Pp[|\M\cap\F| < s-1] = \Pp[|\M\cap\F| \neq s-1] - \Pp[|\M\cap\F| > s-1]$, 

$$\Pp[|\M\cap\F| > s-1] \geq \frac{\Pp[|\M\cap\F| \neq s-1]}{n - s + 2} \geq \frac{\Pp[|\M\cap\F| \neq s-1]}{n}.$$
\end{proof}

The assumption implies that for each $i\in [s-M+1,s]$ we have $$p_{\ff_i} = \Pp[|\F_i \cap \M| \neq s-1] \leq n\Pp[|\F_i \cap \M| > s-1] < \frac{1}{4n^2} \leq \min\Big\{\frac{1}{4ns}, \frac{1}{4kn}\Big\}.$$
For the last inequality we use that $s \leq n$ trivially and that $k \leq n$ for $s$ sufficiently large by assumptions of the theorem. 

Thus, Theorem~\ref{antic} implies that there are $s-1$ parallel hyperplanes each containing at least $(1-2p_{\ff_i})n^{k-1}$ elements of $\F_i$ (i.e. $\ff_i$-fat hyperplanes), and any other hyperplane contains at most $2p_{\ff_i} n^{k-1} + (s-1) n^{k-2}$ elements of $\F_i$. Suppose that $\ff_i$-fat hyperplanes are not the same for all $i\in [s-M+1,s]$. Let $\m H$ be a hyperplane that is not $\ff_i$-fat for some $i\in [s-M+1,s]$, but is $\ff_j$-fat for some other $j \in [s-M+1,s]$. Then $|\bigcup_{i = s-M+1}^s\F_i| \geq |\ff_i \cup (\ff_j\cap \m H)| = |\ff_i| + |\ff_j\cap \m H| - |\ff_i \cap (\ff_j\cap \m H)| \geq |\ff_i| + |\ff_j\cap \m H| - |\ff_i \cap \m H|\geq (s-1)n^{k-1} + (1-\frac{2}{4sn})n^{k-1} - (\frac{2}{4sn} n^{k-1} + (s-1)n^{k-2})$. Here we use $p_\ff \leq \frac{1}{4ns}$. Then $\E[|\bigcup_{i = s-M+1}^s \F_i \cap \M|] = \frac{|\bigcup_{i = s-M+1}^s\F_i|}{n^{k-1}} \geq (s - 1) +\frac 1s- \frac{1}{sn}$. Thus $$(s - 1) +\frac 1s- \frac{1}{sn} \leq \E[|\bigcup_{i = s-M+1}^s \F_i \cap \M|]  \leq s-1 + n \Pp[|\bigcup_{i = s-M+1}^s \F_i \cap \M| > s-1],$$ and thus $\Pp[|\bigcup_{i = s-M+1}^s \F_i \cap \M| > s-1] \geq \frac{1}{sn}-\frac 1{sn^2}$, which contradicts the assumption.

Thus we can assume that 
\begin{equation*}
    |\F_i \cap \m H_1(t)| \geq \Big(1-\frac{1}{2ns}\Big)n^{k-1}\text{ for any }i\in [s-M+1,s]\text{ and any }t \in [s-1],
\end{equation*} 
and we get a rainbow matching by Claim~\ref{cl:str-to-matching}.

\vspace{2mm}

\textbf{Case 2:} $n \geq s^5$.

We will use the same spread approximation  as in Section~\ref{sec5} and the  notation introduced there. Take the same parameter $r$ as in Section~\ref{sec5}, $r= 2^{5} s\log_2(sk)$. Note that the inequality on $k$ from the statement of the theorem and the assumption $n\ge s^5$ imply that for sufficiently large $s$ we have $r^3<\frac n{2s}$. By Lemma~\ref{lem-spred-modif}, if our families $\ff_1, ..., \ff_s$ are cross-dependent, then there exist families $\m S_1, ..., \m S_s \subset 2^{[k]\times[n]}$ and $\ff_i' \subseteq \ff_i$, such that:
\begin{enumerate}
    \item $\F_i \subset \bigcup_{S\in \Sf_i}\m A[S] \cup \F_i'$; 
    \item $|\F_i'|\leq r^3 n^{k-3}\le \frac{n^{k-2}}{2s}$;
    \item $\Sf_i$'s are cross-dependent;
    \item for each $S \in \m S_i$, we have $|S| \leq 2$;
    \item for $\m S_i^{(2)} := \{S \in \Sf_i\, :\, |S| = 2\}$ we have $|\m S_i^{(2)}| \leq 4(s-1)^2$.
\end{enumerate}

As in Section~\ref{sec5}, for $j=1,2$ we denote $\m S_i^{(j)} := \{S \in \Sf_i\, :\, |S| = j\}$. 

We have $|\bigcup_{S\in \Sf_i^{(2)}}\m A[S]| + |\ff_i'| \leq 4s^2 n^{k-2} + \frac{n^{k-2}}{2s} < n^{k-1}$
for sufficiently large $s$. So, in case $\emptyset \notin \m S_i$, to satisfy the first condition in the list above, we should have $|\m S_i^{(1)}| \geq i$ for $i \in [s-M]$ and $|\m S_i^{(1)}| \geq s-1$ for $i \in [s-M+1,s]$. Then, in case that $\emptyset \in \m S_i$ for at least one $i$, the families $\m S_1,\ldots, \m S_s$ cannot be cross-dependent: we can sequentially choose distinct one-element sets from $\m S_i$'s for all $i$'s such that $\emptyset \notin \m S_i$, and empty sets for other $i$'s, what will clearly give us a cross-matching. So $\emptyset \notin \m S_i$ for all $i \in [s]$. 

Suppose that $|\bigcup_{i=1}^s \m S_i^{(1)}| \geq s$. Then Hall's lemma easily implies that we can choose distinct one-element sets from $\m S_i$'s forming a cross-matching. Consequently, there exist $s-1$ elements $(j_1, a_1), ..., (j_{s-1},a_{s-1}) \in [k]\times[n]$ such that $\m S_i^{(1)} = \{\{(j_1, a_1)\}, ..., \{(j_{s-1},a_{s-1})\}\}$ for all $i\in [s-M+1, s]$ and $\m S_i^{(1)} \subset \{\{(j_1, a_1)\}, ..., \{(j_{s-1},a_{s-1})\}\}$ for all $i\in [s-M]$. Then we w.l.o.g. assume that no set from  $\Sf_i^{(2)}$ for $i \in [s-M+1, s]$ contains  $(j_t,a_t)$, $t\in [s-1]$. Indeed, this does not affect the families $\m F_i[\m S_i]$, $i\in[s-M+1,s]$. Suppose that still for some $\ell \in [s-M+1, s]$ we have $\Sf_\ell^{(2)} \neq \emptyset$. Then we can sequentially choose distinct one-element sets from all $\Sf_i$ except $i = \ell$, and then add an element from $\Sf_\ell^{(2)}$. Since all one-element sets are of the form $\{(j_t,a_t)\}$, $t\in[s-1],$ and the element from $\Sf_\ell^{(2)}$ contains none of these elements, we obtain a cross matching. So for all $i\in [s-M+1, s]$ we have $\Sf_i^{(2)} = \emptyset$. Now suppose that for some $t, u \in [s-1]$ we have $j_t \neq j_u$. Then for $i\in [s-M+1,s]$ $|\ff_i| \leq |\m A[\Sf^{(1)}_i]| + |\F_i'| \leq (s-1)n^{k-1} - n^{k-2} + \frac{n^{k-2}}{2s} < (s-1)n^{k-1}$ for sufficiently large $s$. We conclude that all $j_t$ are equal, and we may assume that $j_t = 1$ and $a_t = t$. 

The argument above implies that for all $i\in [s-M+1,s]$ we have $\F_i \subset [s-1]\times[n]^{k-1}\cup \F_i'$. Thus, $|\ff_i\backslash [s-1]\times[n]^{k-1}|\leq |\ff_i'|\le \frac{n^{k-2}}{2s}$. Since $|\F_i| > (s-1)n^{k-1}$, we have $|[s-1]\times[n]^{k-1}\backslash \ff_i| < |\ff_i\backslash [s-1]\times[n]^{k-1}|$. In particular, for any $t \in [s-1]$ we have $|\m H_1(t)\backslash \ff_i| < \frac{n^{k-2}}{2s}$, and so $|\m H_1(t)\cap \ff_i| \geq \big(1 - \frac{1}{2ns}\big)n^{k-1}$. This gives us condition \eqref{assconc}, and thus there again exists a rainbow matching by Claim~\ref{cl:str-to-matching}.

\vspace{2mm}

\textbf{Case 3:} $\Pp[|\bigcup_{i = s-M+1}^s \F_i \cap \M| > s-1] \geq \frac{1}{4n^3}$ and $n \leq s^5$.

As usual, let $\mathcal{M}$ be a uniformly random perfect matching in $[n]^k$.
Consider a bipartite graph $G_\M$ with one part being $\{\mathcal{F}_i\}_{i\in [s]}$, the other part being $\mathcal{M}$, and the edges connecting each  family $\mathcal{F}_i$ to the elements of $\mathcal{M}$ that it contains, i.e., for $i \in [s]$, $F \in \M$, $(\F_i, F) \in E(G_\M)$ iff $F \in \F_i$. Clearly, the degrees of $\F_i$'s in this graph are $deg(\F_i) = |\F_i \cap \M|$.

Consider the following two events:

\begin{itemize}
    \item $A_1$: for every $i \in [s-M]$ we have  $deg(\mathcal{F}_i) \geq i + \frac{M}{2}$;
    \item $A_2$: for every $i \in [s-M+1,s]$ we have $deg(\mathcal{F}_i) \geq M$.
\end{itemize}

Theorem~\ref{conc} implies that $A_1$ and $A_2$ are very likely to hold simultaneously, namely we have

\begin{cla}\label{clap} We have
    $$p := \mathbb{P}[\bar{A_1}\vee \bar{A_2}] \leq 4\,s^{-\frac{C^2}{20}+1}$$
    for $s\ge s_0(C)$.
\end{cla}


\begin{proof}
    By Theorem~\ref{conc} with $\alpha = \frac{i + C\sqrt{s\log s}}{n} \leq \frac{s-1}{n}$, $\delta = -1$ and $\lambda = C\sqrt{s\log s}/4 - 1/2 \leq (C\sqrt{s\log s} - M/2)/2$, for any $i \in [s-M]$  we have

$$\mathbb{P}\left[deg(\mathcal{F}_i) < i + \frac{M}{2}\right] \leq 2\,\exp \left(-\frac{(C\sqrt{s\log s}/4 - 1/2)^2}{(s-1)/2 + 2(C\sqrt{s\log s}/4 - 1/2)}\right) \leq 2\,s^{-\frac{C^2}{20}},$$
where the last inequality is valid for $10\le C\le \sqrt {s/\log s}.$

Similarly, by Theorem~\ref{conc} with $\alpha = \frac{s-1}{n}$, $\delta = -1$ and $\lambda = s/4 \leq (s-1-M)/2$ (this inequality is valid for $C \leq \sqrt {s/\log s}$ and $s \geq 12$) for any $i \in [s-M+1, s]$ we have

$$\mathbb{P}[deg(\mathcal{F}_i) < M] \leq 2\,e^{-\frac{s}{16}}.$$

Then we have  $$p := \mathbb{P}[\bar{A_1}\vee \bar{A_2}] \leq s\cdot 2\,s^{-\frac{C^2}{20}} + s\cdot 2\,e^{-\frac{s}{16}} \leq 4\,s^{-\frac{C^2}{20}+1},$$ where the last inequality holds for  $s>C^2$.
\end{proof}

\vspace{2mm}

Recall that we assumed towards contradiction that the families $\ff_1,\ldots, \ff_s$ are cross-dependent, and thus there is no matching in $G_\M$ that covers all $\mathcal{F}_i$'s. Thus, the condition of Hall's lemma cannot be satisfied. Specifically, there is a set $B \subset [s]$ such that $\{\F_i\}_{i\in B}$ violates the condition, that is, the size of the set of neighbours of $\{\F_i\}_{i\in B}$ is less than $|B|$. In particular, for each $i\in B$ we have  $|\F_i\cap \M| = deg(\F_i) < |B|$.

Let us define a random variable $r_\mathcal{M}$. If both $A_1$ and $A_2$ hold, put $r_\mathcal{M} := |B\cap [s-M+1, s]|$. If either $\bar{A_1}$ or $\bar{A_2}$ holds, set $r_\mathcal{M} = 0$. Let us analyze the structure of $B$ in the assumption that $A_1,A_2$ hold.

Assume that $B\cap[s-M] = \emptyset$. Since $A_2$ holds, for any $i \in B$ we have $deg(\F_i) \geq M \geq |B|$, and thus $B$ cannot violate Hall's condition. Thus $B\cap[s-M] \neq \emptyset$.

Assume next that $r_\M < M/2$. Then $A_1$ implies that for $j = \max\{i : i\in B\cap[s-M]\}$ we have $deg(\mathcal{F}_j) \geq j + \frac{M}{2} \geq |B|$, and thus $B$ satisfies Hall's condition, a contradiction. Overall, if $A_1$ and $A_2$ hold, $r_\M \geq M/2$.

Following \cite{KKconc}, define for $i\in [M]$:

$$\zeta_i := |\mathcal{F}_{s - M + i} \cap \mathcal{M}| - s + 1.$$

By the definition of $B$, for each $j\in B\cap [s-M+1,s]$ we have $deg(\ff_j)\le |B| - 1\le s - M + r_\mathcal{M} - 1$, and thus for the corresponding $i$ we have $\zeta_i \leq - M + r_\mathcal{M} \leq 0$.

Summing this over all $j$ from $B$, we get the following for each  $\mathcal{M}$:
$$\sum_{i=1}^M\zeta_i(\mathcal{M}) \cdot {\rm I}[\zeta_i(\mathcal{M}) \leq 0] \leq r_\mathcal{M}(r_\mathcal{M} - M).$$
Note also that this trivially holds when $r_{\m M}=0$ (in particular, when $A_1$ or $A_2$ are not valid).
It was shown in \cite{KKconc} that $\mathbb{E}[\zeta_i|\zeta_i > 0] \leq 3.7\sqrt{s\log s} =: x$. So we have
{\small \begin{align*}
    0 &< \mathbb{E}\left[\sum_{i=1}^M \zeta_i\right] = \sum_{i=1}^M \mathbb{E}[\zeta_i|\zeta_i > 0] \mathbb{P}[\zeta_i > 0] + \sum_{i=1}^M \mathbb{E}[\zeta_i|\zeta_i \leq 0] \mathbb{P}[\zeta_i \leq 0] \\ &\leq x\sum_{i=1}^M \mathbb{P}[\zeta_i > 0] + \mathbb{E}[r_\mathcal{M}(r_\mathcal{M} - M)] \leq x\mathbb{E}[M - r_\mathcal{M}] + \mathbb{E}[r_\mathcal{M}(r_\mathcal{M} - M)] =  \mathbb{E}[(r_\mathcal{M} - M)(r_\mathcal{M} - x)].
\end{align*}}
The first inequality above is valid since  $\mathbb{E}[\zeta_i] = \mathbb{E}[|\F_{s-M+i}\cap\M|] -s + 1 = \frac{1}{n^{k-1}} > 0$.

If both $A_1$ and $A_2$ hold for $\M$, then $r_\M \geq \frac{M}{2}\geq  \frac{C}{2}\sqrt{s\log s}$. Since we always have $r_\mathcal{M}\le M$, the expression $(r_\mathcal{M} - M)(r_\mathcal{M} - x)$ is negative when $A_1,A_2$ hold, and we have $(r_\mathcal{M} - M)(r_\mathcal{M} - x)\le (\frac{C}{2}\sqrt{s\log s} - x)(r_\mathcal{M} - M)$. If either $A_1$ or $A_2$ does not hold, we have $(r_\mathcal{M} - M)(r_\mathcal{M} - x)= xM$ by definition. Thus, we can continue the displayed chain of inequalities and upper-bound   $\mathbb{E}[(r_\mathcal{M} - M)(r_\mathcal{M} - x)]$ as follows
\begin{align}
 \notag   0 &< (C/2 - 3.7)\sqrt{s\log s} \,\,\mathbb{E}[(r_\mathcal{M} - M)|A_1, A_2](1-p) + xM p  \\ 
\notag &\leq
    -(C/2 - 3.7)\sqrt{s\log s}\,\,\mathbb{P}[(r_\mathcal{M} < M)|A_1, A_2](1-p) + xM p \\ 
\notag &\leq -(C/2 - 3.7)\sqrt{s\log s}\,\,(\mathbb{P}[(r_\mathcal{M} < M)] - p) + xM p \\
 \label{eq1}   &= -(C/2 - 3.7)\sqrt{s\log s}\,\,\mathbb{P}[(r_\mathcal{M} < M)] + p\left(xM + (C/2 - 3.7)\sqrt{s\log s}\right).
\end{align}

Recall that $r_\M \leq M$ and equality may hold only if $deg(\F_i) \leq s-1$ for all $i\in [s-M+1, s]$, and, moreover, $|\bigcup_{i = s-M+1}^s \F_i \cap \M| \leq s-1$. Thus, $$\mathbb{P}[r_\mathcal{M} < M] \geq \Pp\left[\Big|\bigcup_{i = s-M+1}^s \F_i \cap \M\Big| > s-1\right]\geq \frac{1}{4n^3} \geq \frac{1}{4 s^{15}}$$
by the conditions of the current case, and inequality \eqref{eq1} together with Claim~\ref{clap} imply that

$$0 < -(C/2 - 3.7)\sqrt{s\log s}\,\,\frac{1}{4 s^{15}} + 4s^{-\frac{C^2}{20}+1}(xM + (C/2 - 3.7)\sqrt{s\log s}).$$

For $C \geq 20$, this gives (for some constant $A>0$):

$$s^{-14,5} \sqrt{\log s} < A s^{-\frac{C^2}{20}+2} \log s \leq A s^{-18} \log s,$$
which is clearly violated for sufficiently large $s$. This completes the proof of the theorem.
\end{proof}

\section{Easy proof for large $n$}\label{sec-easy}

Now we shall treat the easy case of large $n$.

\begin{thm}\label{st5}
    The sequence $\{\min(s-1, i)n^{k-1}\}_i$ is satisfying for $n \geq k^2 s^2$.
\end{thm}

\begin{proof}

The argument is a version of the original Erd\H os's argument \cite{E} to prove Erd\H os Matching Conjecture for $n>n_0(s,k)$, as it is given in \cite{HLS}. It is based on the following claim.

\begin{cla}
    If families $\F_1, ..., \F_s$ are cross-dependent and for some $i \in [s]$, $j\in [k]$ and $a\in[n]$ we have $|\m H_j(a)\cap \F_i| > (s-1)(k-1)n^{k-2}$, then the families $\F_1, ...,\F_i\cup \m H_j(a),...,\F_s$ are also cross-dependent.
\end{cla}

\begin{proof}
    Arguing indirectly, assume that $\F_1, ...,\F_i\cup \m H_j(a),...,\F_s$ contain a cross-matching $F_1\in \ff_1,..,F_s\in \ff_s$. If $F_i \notin \m H_j(a)\backslash \F_i$, then it is a cross-matching in the initial families, a contradiction. Thus $F_i \in \m H_j(a)\backslash \F_i$ and so $F_\ell \notin \m H_j(a)$ for all $\ell \neq i$. We have $|\m H_j(a)\cap \{X: X\cap F_\ell \neq \emptyset\}| \leq (k-1)n^{k-2}$ for all $\ell \neq i$. Consequently, $|\m H_j(a)\cap \bigcup_{\ell\in [k]\setminus\{i\}}\{X: X\cap F_\ell \neq \emptyset\}| \leq (s-1)(k-1)n^{k-2} < |\m H_j(a)\cap \F_i|$, and hence there is some $F_i' \in \m H_j(a)\cap \F_i$ that is disjoint from all $F_\ell$, $\ell \neq i$, and thus $F_1,...,F_i',...,F_s$ is a cross-matching in $\F_1, ..., \F_s$, a contradiction.
\end{proof}

    By the claim, we can w.l.o.g. assume that $\F_i = \hat{\F_i}\sqcup\F_i'$, where $\hat{\F_i}$ is a union of $t_i$ hyperplanes and $\F_i'$ satisfies the following for all $j\in [k]$ and $a\in[n]$: $$|\m H_j(a)\cap \F_i'| \leq (s-1)(k-1)n^{k-2}.$$

For each $i < s$, we either have $t_i \geq i$ or $$|\ff_i'| \geq |\ff_i|-t_in^{k-1}\ge n^{k-1} \geq s^2k^2 n^{k-2}.$$ Similarly, for $i=s$ we either  have $t_s \geq s-1$ or $|\ff_s'| \geq n^{k-1} \geq s^2 k^2 n^{k-2}$.

Let $u_1 < ... < u_m < s$ be the indices such that  $t_{u_j} \geq u_j$, $j\in[m]$. For each $j\in [m]$ we can choose a hyperplane $\m H_{u_j}\subset \ff_{u_j}$ and, moreover, all these hyperplanes are different. We have $|\bigcup_{j\in[m]} \m H_{u_j}| \leq (s-1)n^{k-1}$, and since $|\F_s| > (s-1)n^{k-1}$, the family $\F_s$ contains a set $F_s$ that is not contained in any of $\{\m H_{u_j}\}_j$. Let us show that we can sequentially choose $F_{s-1}, F_{s-2},...,F_{1}$ such that, first, $F_i \cap F_j = \emptyset$ for all $j > i$ and, second, $F_i \notin \m H_{u_j}$ for all $j$ such that $u_j \neq i$. Indeed, suppose we have chosen $F_{s-1},\ldots, F_{i+1}$ already, and we wish to pick an appropriate $F_i$. Assume first that $i \notin \{u_1,...,u_m\}$. Then $|\F'_i| \geq s^2 k^2 n^{k-2}$. Each hyperplane contains at most $(s-1)(k-1)n^{k-2}$ elements of $\F_i'$ and thus any set intersects at most $k(s-1)(k-1)n^{k-2}$ elements of it. Thus, $\F'_i$ contains at most $(s-1)^2(k-1)n^{k-2} + (s-1)^2k(k-1)n^{k-2} \leq s^2k^2 n^{k-2}$ elements that do not satisfy the desired conditions. Consequently, there is at least one that suits. If $i = u_j$ for some $j\in[m]$, then each $\m H_{u_\ell}$ with $\ell \neq j$ intersects $\m H_{u_j}$ in at most $n^{k-2}$ sets. By construction, the sets $F_{i+1}, ..., F_s$ do not lie in $\m H_{u_j}$, and thus it contains at most $ksn^{k-2}$ elements that intersect one of $F_{i+1}, ..., F_s$. Since $|\m H_{u_j}| = n^{k-1} > s n^{k-2} + ksn^{k-2}$, there is a set $F_{u_j}\in \m H_{u_j} \subset \F_{u_j}$ that satisfies both conditions. Clearly, $F_1, ..., F_s$ is a cross-matching.
\end{proof}

\section{Nullstellensatz: proof of Theorem~\ref{thmmain4}}\label{sec6}

In this section, we shall use Alon's Combinatorial Nullstellensatz to prove Theorem~\ref{thmmain4}. Recall the statement of the Combinatorial Nullstellensatz:

\begin{thm}[Alon \cite{Al}]\label{thmnull}
    Let $f(x_1,...,x_n)$ be a polynomial over a field $\mathbb{F}$ of degree $k$. Suppose that the coefficient of a monomial $x_1^{k_1}x_2^{k_2}...x_n^{k_n}$ in $f$ is nonzero and $k_1 + ... + k_n = k$. Then for any $A_1, ..., A_n \subset \mathbb{F}$, such that $|A_i| > k_i$, there exist $a_1 \in A_1$, ..., $a_n \in A_n$ such that $f(a_1, ..., a_n) \neq 0$.
\end{thm}

\begin{proof}[Proof of Theorem~\ref{thmmain4}]
Let $\mathbb{F} = \mathbb{Z}_p(\alpha)$ be a quadratic extension of $\mathbb{Z}_p$, where $\alpha$ is a root of $x^2 = a$ with $a$ being a quadratic non-residue in $\mathbb{Z}_p$. In particular, every element of $\mathbb{F}$ can be written uniquely in the form $u + \alpha v$ with $u,v \in \mathbb{Z}_p$, i.e. the map $\varphi: [p]^2 \to \mathbb{F}$ given by $\varphi(i,j) = i + \alpha j$ is a bijection.

\begin{cla}\label{clazp}
$(\varphi((i_1,j_1)) - \varphi((i_2,j_2)))^2 \in \mathbb{Z}_p$ if and only if $i_1 = i_2$ or $j_1 = j_2$.
\end{cla}

\begin{proof}
We compute
\[
(\varphi((i_1,j_1)) - \varphi((i_2,j_2)))^2 
= ((i_1-i_2)^2 + a(j_1-j_2)^2) + 2 \alpha (i_1-i_2)(j_1 - j_2).
\]

Since $\{1,\alpha\}$ is a basis of $\mathbb{F}$ over $\mathbb{Z}_p$, this element lies in $\mathbb{Z}_p$ if and only if the coefficient of $\alpha$ is zero, i.e.,
$(i_1-i_2)(j_1 - j_2) = 0,$
which is equivalent to $i_1 = i_2$ or $j_1 = j_2$.
\end{proof}

Consider the following polynomial over $\mathbb{F}$:
\begin{equation}\label{eqp}
P(x_1, ..., x_s) = \prod_{1\leq i<j \leq s} \prod_{q \in \mathbb{Z}_p} \big((x_j - x_i)^2 - q\big).
\end{equation}
For each pair $i<j$, the inner product runs over all $p$ elements of $\mathbb{Z}_p$, and hence contributes degree $2p$. Therefore,
\[
\deg P = 2p \binom{s}{2} = p s (s-1).
\]

Observe that the monomials of maximal total degree arise by taking, for each pair $i<j$, the leading term $(x_j - x_i)^{2p}$. Thus, the homogeneous component of highest degree is
$\prod_{1\leq i<j \leq s} (x_j - x_i)^{2p}.$
Using the Frobenius identity in characteristic $p$, we have $(x_j - x_i)^p = x_j^p - x_i^p$, and hence
\[
\prod_{1\leq i<j \leq s} (x_j - x_i)^{2p}
= \prod_{1\leq i<j \leq s} (x_j^p - x_i^p)^2.
\]

Now let $\mathcal{F}_1, ..., \mathcal{F}_s \subset [p]^2$ be bipartite graphs such that $|\mathcal{F}_i| > p f_i$ for each $i$. Since $\varphi$ is a bijection, we have
\[
|\varphi(\mathcal{F}_i)| = |\mathcal{F}_i| > p f_i.
\]

Next, consider the coefficient of the monomial $x_1^{pf_1}\cdots x_s^{pf_s}$ in the leading part of $P$, i.e.
$\prod_{1\leq i<j \leq s} (x_j^p - x_i^p)^2$. This coefficient coincides with the coefficient of $x_1^{f_1}\cdots x_s^{f_s}$ in
$\prod_{1\leq i<j \leq s} (x_j - x_i)^2$,
and hence is non-zero in $\mathbb{F}$. Therefore, the polynomial $P$ satisfies the condition of Theorem~\ref{thmnull} with $k_i = p f_i$.

Applying the Combinatorial Nullstellensatz to the sets $\varphi(\mathcal{F}_1),...,\varphi(\mathcal{F}_s)$, we obtain elements $X_1 \in \mathcal{F}_1$, ..., $X_s \in \mathcal{F}_s$ such that
\[
P(\varphi(X_1), ..., \varphi(X_s)) \neq 0.
\]

By the definition of $P$ in \eqref{eqp}, this implies that for every pair $i<j$,
\[
(\varphi(X_j) - \varphi(X_i))^2 \notin \mathbb{Z}_p.
\]
Finally, by Claim~\ref{clazp}, this implies that $X_i$ and $X_j$ do not intersect for $i \ne j$. Hence, $X_1,...,X_s$ is the desired cross-matching.
\end{proof}

\section{Discussion}
In this paper, we focused on two types of sequences $f_1,\ldots, f_s$. The first type is an arithmetic progression $f_i = m+(i-1)d$, and the second type is a `truncated' arithmetic progression $f_i = \min\{m+(i-1)d, (s-1)n^{k-1}\}.$ For both types of sequences, we had $d = n^{k-1},$ which appears to be the right value to take. The optimal value of $m= m(n,s,k)$ would be very interesting to determine. 
\begin{pro}
    Determine the smallest $m = m(n,s,k)$ such that the sequence $f_i = m+(i-1)n^{k-1}$ is satisfying.
\end{pro}
\begin{pro}
    Determine the smallest $m = m(n,s,k)$ such that the sequence $f_i = \min\{m+(i-1)n^{k-1}, (s-1)n^{k-1}\}$ is satisfying.
\end{pro}
On the other hand, the sequences from examples in Section~\ref{sec2} all have the form $f_1 = \ldots = f_\ell = f,$ $f_{\ell+1} = \ldots = f_s = g.$ Thus, the following problem is natural.
\begin{pro}
    Determine the minimal values $f = f(n,k,s,\ell)$ and $g = f(n,k,s,\ell)$ such that the sequence $f_1 = \ldots = f_\ell = f,$ $f_{\ell+1} = \ldots = f_s = g$ is satisfying.
\end{pro}
The case $g = (s-1)n^{k-1}$ is of particular interest.
\begin{pro}
    Determine the minimal value $f = f(n,k,s,\ell)$  such that the sequence $f_1 = \ldots = f_\ell = f,$ $f_{\ell+1} = \ldots = f_s = (s-1)n^{k-1}$ is satisfying.
\end{pro}
It is not unlikely that the examples from Section~\ref{sec2} are best possible.

Satisfying sequences for $k=1$ are easy to describe: it is necessary and sufficient that $f_i\ge i-1$. (We of course suppose that $f_1\le f_2\le\ldots \le f_s$.) The problem for the graph case $k=2$ and arbitrary sequences seems to be non-trivial and very interesting. 
\begin{pro}
    Describe the set of  all satisfying sequences $f_1,\ldots, f_s$ for $k=2$ and arbitrary $s$.
\end{pro}
If it is resolved, it should give good insight on satisfying sequences for arbitrary $k$ and sufficiently large $n$.

Finally, it would be very interesting (although much harder) to treat analogous questions for families $\ff_1,\ldots, \ff_s\subset {[n]\choose k}$. 

\section{Acknowledgements} We thank the referees for their numerous helpful comments.  The research of Andrey Kupavskii was supported by the
Ministry of Economic Development of the Russian Federation (agreement with
MIPT No. 139-15-2025-013, dated June 20, 2025, IGK 000000C313925P4B0002).

\end{document}